\documentclass{amsart}   	
\usepackage{amssymb,amsthm}
\usepackage{amsthm}
\usepackage{graphicx}
\usepackage{enumitem}
\usepackage{tikz}
\usepackage{tikz-cd}
\usetikzlibrary{calc}
\usetikzlibrary{decorations.pathmorphing}
\usepackage{hyperref}
\usepackage{adjustbox}
\usepackage{mathtools}
\usepackage[all]{xy}
\DeclarePairedDelimiter\ceil{\lceil}{\rceil}
\DeclarePairedDelimiter\floor{\lfloor}{\rfloor}
\tikzset{curve/.style={settings={#1},to path={(\tikztostart)
    .. controls ($(\tikztostart)!\pv{pos}!(\tikztotarget)!\pv{height}!270:(\tikztotarget)$)
    and ($(\tikztostart)!1-\pv{pos}!(\tikztotarget)!\pv{height}!270:(\tikztotarget)$)
    .. (\tikztotarget)\tikztonodes}},
    settings/.code={\tikzset{quiver/.cd,#1}
        \def\pv##1{\pgfkeysvalueof{/tikz/quiver/##1}}},
    quiver/.cd,pos/.initial=0.35,height/.initial=0}

\tikzset{tail reversed/.code={\pgfsetarrowsstart{tikzcd to}}}
\tikzset{2tail/.code={\pgfsetarrowsstart{Implies[reversed]}}}
\tikzset{2tail reversed/.code={\pgfsetarrowsstart{Implies}}}
\tikzset{no body/.style={/tikz/dash pattern=on 0 off 1mm}}

\graphicspath{ {./images/} }

\newtheorem{theorem}{Theorem}[section]

\newtheorem{lemma}[theorem]{Lemma}
\newtheorem{corollary}[theorem]{Corollary}

\theoremstyle{definition}
\newtheorem{definition}[theorem]{Definition}

\newtheorem{question}[theorem]{Question}

\newcommand{\qedno}{\null\nobreak\hfill\ensuremath{\square}}

\usepackage[utf8]{inputenc}

\title{Structure of a Maximal Total Independent Set}
\author{Lewis Stanton}

\email{lrs1g18@soton.ac.uk}


\begin{document}
\maketitle
\begin{abstract}
Let $G$ be a simple, connected and finite graph of order $n$. Denote the independence number, edge independence number and total independence number by $\alpha(G), \alpha'(G)$ and $\alpha''(G)$ respectively. This paper establishes a relation between $\alpha''(G)$ with $\alpha(G)$, $\alpha'(G)$ and $n$. It also describes the possible structures of a total independent set of a given size.
\end{abstract}
\section{Introduction}
\label{intro}

The independence number (denoted $\alpha(G)$) of a graph $G$ is an extensively studied invariant in graph theory. There are a variety of bounds and computer algorithms which can be used to find the independence number and to construct an independent set of vertices. The same is true for a variant of the independence number known as the edge independence number (denoted $\alpha'(G)$). This is a similar invariant to the independence number but is defined on the edge set of the graph as opposed to the vertex set of a graph. The edge independence number is also well understood and is completely determined for certain classes of graphs such as Hamiltonian graphs. 

The total independence number (denoted $\alpha''(G)$) is another variation of the definition of the independence number, being defined on the union of the vertex and edge sets of $G$. While its value can be calculated for all graphs in terms of the edge domination number of a graph, there is not much known in terms of its relationship with other invariants and the structure of a total independent set containing a given number of elements.

In this paper, we explore the total independence number and its relationship with the independence number and edge independence number. We show in Theorem \ref{mainbound} that for a simple, connected and finite graph of order $n$, $\alpha''(G) \leq n + \floor*{\frac{\alpha(G)-n + 2\alpha'(G)}{2}} - \alpha'(G)$. Corollary \ref{otherbound} gives another bound which is that $\alpha''(G) \leq \frac{n + \alpha(G)}{2}$. In Section \ref{structure}, we explore the possible structures a total independent set of a given size can have, namely the possible number of edges and vertices it can contain.

\section{Basic Definitions}
\label{basicdefns}

In this paper, we will use the definitions from Gross, Yellen and Zhang \cite{GYZ}. A graph $G$ is \textit{finite} if its vertex set $V(G)$ and edge set $E(G)$ are finite. The \textit{order} of a graph $G$ is the number of vertices of $G$. We denote $|V(G)|$ as order$(G)$. Similarly, the \textit{size} of a graph $G$ is the number of edges of $G$. We denote $|E(G)|$ as size$(G)$.  A graph $G$ is \textit{simple} if any edge has two distinct end vertices and there is at most one edge between any two distinct vertices. Finally, a graph $G$ is \textit{connected} if for any two distinct vertices $x,y \in V(G)$ there exists a walk in $G$ between $x$ and $y$. \[\begin{tikzpicture} 
	   \draw (-0.1,-0.1) -- (0.5,0.6) -- (0.5,-0.8) -- (-0.1,-0.1);
     \draw (0.5,0.6) -- (1.1,-0.1) -- (0.5,-0.8);
		 \draw [fill] (-0.1,-0.1) circle [radius=0.03];
		 \draw [fill] (0.5,0.6) circle [radius=0.03];
		 \draw [fill] (0.5,-0.8) circle [radius=0.03];
		 \draw [fill] (1.1,-0.1) circle [radius=0.03];
	   \node at (-0.35,-0.1) {1};
		 \node at (0.5,0.85) {2};
		 \node at (0.5,-1.05) {3};
		 \node at (1.35,-0.1) {4};
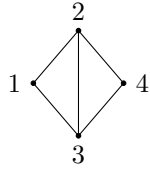
          \node at (0.5,-1.5) {A simple, connected graph of order $4$ and size $5$};
  \end{tikzpicture}\]

We will restrict our attention to simple, connected and finite graphs; however, many of the results leading up to Theorem \ref{mainbound} may hold in a broader context. Let $G$ be a simple, connected and finite graph of order $n$. We can write an edge $e \in E(G)$ in terms of the two vertices that it joins together. These are known as the \textit{end vertices} of the edge, and we denote the edge $e$ with end vertices $x$ and $x'$ as $x \cdot x'$, where $x,x' \in V(G)$. There are various notions of adjacency in a graph. Two vertices $v, v' \in V(G)$ are adjacent if there is an edge $e \in E(G)$ such that $e = v \cdot v'$. Two edges $e,e' \in E(G)$ are adjacent if they share an end vertex. Finally, a vertex $v \in V(G)$ and $e \in E(G)$ are adjacent if $v$ is an end vertex of $e$. 

With these definitions of adjacency, we aim to find subsets of the vertex set or edge set such that none of the elements of the subset are pairwise adjacent. We can do this as the basic structure of the definitions of adjacency are the same; the only difference is the elements that are being considered. 

\begin{definition}
A subset $A \subseteq V(G)$ is an \textit{independent set} if the vertices in $A$ are pairwise non-adjacent. The \textit{independence number} of $G$ is \[\alpha(G) = \max\{\left\lvert A \right\rvert \:|\: \text{ A is an independent set of } G\}.\]
\end{definition}

\begin{definition}
A subset $A' \subseteq E(G)$ is an \textit{edge independent set} if the edges in $A'$ are pairwise non-adjacent. The \textit{edge independence number} is \[\alpha'(G) = \max\{\left\lvert A' \right\rvert \:|\: A' \text{ is an edge independent set of } G\}.\]
\end{definition}

\begin{definition}
A subset $A'' \subseteq V(G) \cup E(G)$ is a \textit{total independent set} if the elements of $A''$ are pairwise non-adjacent. The \textit{total independence number} is \[\alpha''(G) = \max\{\left\lvert A'' \right\rvert \:|\: A'' \text{ is a total independent set of } G\}.\] 
\end{definition}

\begin{equation}\begin{tikzpicture}\label{independencenumberfigure} 
	   \draw (-0.1,-0.1) -- (0.5,0.6) -- (0.5,-0.8);
     \draw (0.5,0.6) -- (1.1,-0.1);
		 \draw [fill] (-0.1,-0.1) circle [radius=0.03];
		 \draw [fill] (0.5,0.6) circle [radius=0.03];
		 \draw [fill] (0.5,-0.8) circle [radius=0.03];
		 \draw [fill] (1.1,-0.1) circle [radius=0.03];
	   \node at (-0.35,-0.1) {1};
		 \node at (0.5,0.85) {2};
		 \node at (0.5,-1.05) {3};
		 \node at (1.35,-0.1) {4};
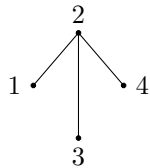
          \node at (0.5,-1.5) {A graph with $\alpha(G) = 3$, $\alpha'(G) = 1$, and $\alpha''(G) = 3$.};
  \end{tikzpicture}\end{equation}

From the definition of the independence number, we can prove a basic bound which we will use later. The \textit{discrete graph} on $n$ vertices is the graph with $V(G) = \{x_1, \cdots, x_n\}$ and $E(G) = \emptyset$.

\begin{lemma}
\label{discreteindep}
    If $G$ is a finite graph of order $n$ that is not the discrete graph, then $\alpha(G)$ < n.
\end{lemma}
\begin{proof}
    Suppose we have an independent set containing $n$ vertices. Since $G$ is not the discrete graph, there exists an edge $x \cdot y$. However, this means that $x$ and $y$ can't be in the independent set containing all of the vertices as they are adjacent.
\end{proof}

We will also use a simple bound on the edge independence number.

\begin{lemma}
\label{basicedgeindepbound}
Let $G$ be a simple, connected and finite graph of order $n$. Then $\alpha'(G) \leq \floor*{\frac{n}{2}}$.
\end{lemma}
\begin{proof}
Let $X$ be an edge independent set of size $\alpha'(G)$. Each edge in $X$ can not share an end vertex with any other edge otherwise they would be adjacent to one another. Therefore, there are $2\alpha'(G)$ vertices contained in $X$ as end points of the edges. Hence, $2\alpha'(G) \leq n$ which implies that $\alpha'(G) \leq \frac{n}{2}$. Since $\alpha'(G)$ is an integer, $\alpha'(G) \leq \floor*{\frac{n}{2}}$ as desired.
\end{proof}

A \textit{Hamiltonian cycle} is a cycle that contains each vertex exactly once. A \textit{Hamiltonian graph} is a graph which contains a Hamiltonian cycle. For Hamiltonian graphs, we obtain a stronger result. 

\begin{lemma}
\label{Hamiltonianedgeindep}
Let $G$ be a simple, connected and finite graph of order $n$ which is Hamiltonian. Then $\alpha'(G) = \floor*{\frac{n}{2}}$.
\end{lemma}
\begin{proof}
Let $C$ be a Hamiltonian cycle in $G$ consisting of vertices $x_1, \cdots, x_n$. We can construct a total independent set by taking every other edge in the Hamiltonian cycle. If $n$ is even, take $x_1 \cdot x_2, x_3 \cdot x_4, \cdots, x_{n-1} \cdot x_n$ as an edge independent set. If $n$ is odd, take $x_1 \cdot x_2, x_3 \cdot x_4, \cdots, x_{n-2} \cdot x_{n-1}$. In both cases, these are edge independent sets containing $\floor*{\frac{n}{2}}$ edges. Therefore, $\alpha'(G) \geq \floor*{\frac{n}{2}}$. However, since $\alpha'(G) \leq \floor*{\frac{n}{2}}$ by Lemma \ref{basicedgeindepbound}, $\alpha'(G) = \floor*{\frac{n}{2}}$.
\end{proof}

We also have a basic upper and lower bound on the total independence number which follows from the definition.

\begin{lemma}
\label{basicbound}
Let $G$ be a simple, connected and finite graph. Then $\alpha''(G) \leq \alpha(G) + \alpha'(G)$.
\end{lemma}
\begin{proof}
Let $X$ be a total independent set of size $\alpha''(G)$. Let $X_v = X \cap V(G)$ be the set of vertices contained in $X$ and $X_e = X \cap E(G)$ be the set of edges contained in $X$. From the definition of $\alpha(G)$ and $\alpha'(G)$, we must have that $|X_v| \leq \alpha(G)$ and $|X_e| \leq \alpha'(G)$ and hence $|X| = |X_v| + |X_e| \leq \alpha(G) + \alpha'(G)$. 
\end{proof}

\begin{lemma}
Let $G$ be a simple, connected and finite graph. Then $\alpha''(G) \geq \max\{\alpha(G),\alpha'(G)\}$.
\end{lemma}
\begin{proof}
Let $X_{\alpha(G)}$ be an independent set of size $\alpha(G)$. This is clearly a total independent set so $\alpha''(G) \geq \alpha(G)$. Similarly, let $X_{\alpha'(G)}$ be an edge independent set of size $\alpha'(G)$. This is clearly also a total independent set so $\alpha''(G) \geq \alpha'(G)$. Combining these together, we find that $\alpha''(G) \geq \max\{\alpha(G),\alpha'(G)\}$.
\end{proof}

These two bounds do not provide much insight into the total independence number. By considering the underlying structure of a total independent set, Theorem \ref{mainbound} will provide an alternative bound on the total independence number in relation to the independence number and the edge independence number. In Section \ref{howgood}, we consider when this bound is strictly better than the bound in Lemma \ref{basicbound}.

There are some connections between dominating sets and independent sets which we will use. 

\begin{definition} 
A subset $D \subseteq E(G)$ is an \textit{edge dominating set} if every edge $e \in E(G) \setminus D$ is adjacent to an edge in $D$. The \textit{edge domination number} of $G$ is $\gamma'(G) = \min\{\left\lvert D \right\rvert \:|\: D$ is an edge dominating set of $G\}$. 
\end{definition}

The edge domination number can be used to determine the total independence number \cite{LZZ}.

\begin{lemma}
\label{edgedominationtotalindep}
Let $G$ be a simple, connected and finite graph. We have that $\gamma'(G) + \alpha''(G) = n$.
\qedno
\end{lemma}

Observe that in (\ref{independencenumberfigure}), $n = 4$ and $\gamma'(G) = 1$, and so Lemma ~\ref{edgedominationtotalindep} implies that $\alpha''(G) = 3$ as expected. The property that we will require is that a largest edge independent set is also an edge dominating set.

\begin{lemma}
\label{edgeindepisdomin}
Let $G$ be a simple, connected and finite graph of order $n$. Let $A' \subseteq E(G)$ be an edge independent set of size $\alpha'(G)$. Then $A'$ is also an edge dominating set. 
\end{lemma}

\begin{proof}
Let $A' \subseteq E(G)$ be a maximal edge independent set of size $\alpha'(G)$. Suppose $A'$ is not an edge dominating set. This implies that there is an edge $e \in E(G)$ such that $e$ is not adjacent to any edge in $A'$. However, this implies that the set $B' = A' \cup \{e\}$ is an edge independent set of size larger than $A'$, contradicting the maximality of $A'$.
\end{proof}

There are two further invariants of graphs in this section which we consider. An \textit{acyclic} graph is a graph without any cycles. A \textit{subgraph} of a simple, connected and finite graph $G$ is a simple. finite graph $H$ for which $V(H) \subseteq V(G)$ and $E(H) \subseteq E(G)$. We denote $H$ being a subgraph of $G$ as $H \subseteq G$. An \textit{induced subgraph} is a subgraph of $G$ which has vertex set $V(H)$, and contains all of the edges of $G$ on the vertices of $H$.

\begin{definition}
The \textit{arboricity} of an undirected graph $G$ is the minimum number of subsets that $V(G)$ can be partitioned into such that each subset induces an acyclic graph. We will denote this $\beta(G)$.
\end{definition}

\begin{definition}
The \textit{edge arboricity} of an undirected graph $G$ is the minimum number of subsets that $E(G)$ can be partitioned into such that each subset induces an acyclic graph. We will denote this $\beta'(G)$.
\end{definition}

These are also both well studied invariants which tell us how dense a graph is. The value of the edge arboricity of a graph is given by the Nash-Williams theorem \cite{CMWZZ}.  

\begin{theorem}
    Let $G$ be a finite, simple and connected graph. The edge arboricity is given by the formula \[\beta'(G) = \max_{H \subset G}\ceil*{\frac{m_H}{n_H}-1}\] where $H$ runs over all subgraphs of $G$ with order strictly greater than $1$, $n_H$ denotes the order of $H$ and $m_H$ denotes the size of $H$.
    \qedno
\end{theorem}

We can use the arboricity and edge arboricity to bound the total independence number \cite{CWZZ}.

\begin{lemma}
Let $G$ be a simple, connected and finite graph. Then \[1+\floor*{\frac{n}{2}} \leq \beta(G) + \alpha''(G) \leq n+1\] and \[1+\floor*{\frac{n}{2}} \leq \beta'(G) + \alpha''(G) \leq n+ \floor*{\frac{(n+1)}{8}}.\]
\qedno
\end{lemma}

We also have some special types of graph which we will consider. The \textit{complete graph} $K_n$, $n \geq 1$ is a graph with $V(G) = \{x_1, \cdots, x_n\}$ and the edge set consists of an edge for every pair of distinct vertices $x_i,x_j$, $1 \leq i < j \leq n$. The \textit{complete bipartite graph} $K_{m,n}$, $n,m \geq 1$ is a graph with $V(G) = A \cup B$ where $A$ and $B$ are disjoint sets with $|A| = m$ and $|B| = n$. The edge set consists of all edges of the form $a \cdot b$ where $a \in A$ and $b \in B$.

In this paper, we will prove a relation between $\alpha''(G)$ with $\alpha(G)$ and $\alpha'(G)$ and use the proof of this to show the possible structures of a total independent set with a given number of elements.

\section{Preliminary Result}
\label{prelimres}

In this Section, we set out some basic results which will help us prove the main theorem. The purpose of these lemmas is to show that the procedure we use to prove the main theorem will always work for all simple, connected and finite graphs. 

\begin{lemma}
\label{vertandedgebound}
For any simple, connected and finite graph $G$ of order $n$, $\alpha(G) + 2\alpha'(G) \geq n$.
\end{lemma}
\begin{proof}
First consider the case $\alpha'(G) = \floor*{\frac{n}{2}}$. Since $n$ is an integer, $\frac{n-1}{2} \leq \floor*{\frac{n}{2}} \leq \frac{n}{2}$ which implies that $n-1 \leq 2 \cdot \floor*{\frac{n}{2}} \leq n$. Therefore, $2\alpha'(G) \geq n-1$. Since $\alpha(G) \geq 1$, it follows that $2\alpha'(G) + \alpha(G) \geq n$.

Now suppose $\alpha'(G) < \floor*{\frac{n}{2}}$.  Note that $\alpha'(G) < \floor*{\frac{n}{2}}$ implies that $2\alpha'(G) < n$, and so $n-2\alpha'(G) \geq 1$. Now let $F$ be a maximal edge independent set of size $\alpha'(G)$. In particular, $F$ contains $2\alpha'(G)$ distinct vertices as end vertices of edges in $F$. Let $X_F$ be the set of vertices which are end vertices of edges in $F$. Note that $|X_F| = 2\alpha'(G)$. As $F$ is a maximal edge independent set, by Lemma \ref{edgeindepisdomin} it is also an edge dominating set. Therefore, for all $e \in E(G)$, $e = x\cdot y$ for some $x \in V(G)$ and $y \in X_F$. Now consider $V(G) \setminus X_F$. By definition, $|V(G) \setminus X_F| = n - 2\alpha'(G) \geq 1$. We want to show that $V(G) \setminus X_F$ is an independent set. 

If $|V(G) \setminus X_F| = 1$ then it is trivially an independent set. If $|V(G) \setminus X_F| > 1$, consider $x,x' \in V(G) \setminus X_F$ where $x,x'$ are distinct. Suppose $x,x'$ are adjacent vertices, then $x \cdot x' \in E(G)$. However, this implies that either $x \in X_F$ or $x' \in X_F$ which is a contradiction. Therefore, none of the vertices in $V(G) \setminus X_F$ can be adjacent to each other so they form an independent set. Hence, $\alpha(G) \geq |V(G) \setminus X_F| = n-2\alpha'(G)$, and so $\alpha(G) + 2\alpha'(G) \geq n$.
\end{proof}

\section{Main Theorem}
\label{mainthm}

With the results in the previous section, we can now prove the main theorem. The basic idea is to partition the total independent sets by the amount of edges contained in them and construct an upper bound for each set in turn. The upper bound therefore on $\alpha''(G)$ will be the biggest of these bounds. For a set $X \subset V(G) \cup E(G)$ in the following, we will use the following notation, $X_v = X \cap V(G)$ and $X_e = X \cap E(G)$.

\begin{theorem}
\label{mainbound}
Let $G$ be a simple, connected and finite graph of order $n$. Then \[\alpha''(G) \leq n + \floor*{\frac{\alpha(G)-n + 2\alpha'(G)}{2}} - \alpha'(G).\]
\end{theorem}
\begin{proof}
Let $G$ be a simple, connected and finite graph and let $\mathcal{X}_G$ be the set of total independent sets of $G$. Partition $\mathcal{X}_G = \mathcal{X}_{G_{\alpha'(G)}} \cup \mathcal{X}_{G_{\alpha'(G)-1}} \cdots \cup \mathcal{X}_{G_{0}}$ where $\mathcal{X}_{G_{i}}$ is the set of total independent sets containing $i$ edges. Let $m' = \floor*{\frac{\alpha(G)-n + 2\alpha'(G)}{2}}$. By Lemma \ref{vertandedgebound}, $m' \geq 0$. Since $G$ is simple, connected and finite, by Lemma \ref{discreteindep} $\alpha(G) < n$, which implies that $\alpha(G) - n < 0$. Therefore, $\alpha(G) - n + 2\alpha'(G) < 2\alpha'(G)$ and so $\frac{\alpha(G)-n + 2\alpha'(G)}{2} < \alpha'(G)$. By definition $m' \leq \frac{\alpha(G)-n + 2\alpha'(G)}{2}$. Hence, $m' < \alpha'(G)$. Observe each vertex can only appear once in a total independent set and we can have a maximum of $\alpha'(G)$ edges in a total independent set and $\alpha(G)$ vertices in the total independent set. 

Consider $X \in \mathcal{X}_{G_{\alpha'(G)-i}}$ with $0 \leq i \leq m'$. Write $X$ as $X = X_v \cup X_e$. By definition of $\mathcal{X}_{G_{\alpha'(G)-i}}$, $|X_e| = \alpha'(G)-i$. This means that $2(\alpha'(G)-i)$ vertices appear in $X_e$ as end vertices. Now, \[n-2(\alpha'(G)-i) = n-2\alpha'(G)+2i \leq n-2\alpha'(G)+2m' \leq n-2\alpha'(G) + 2\alpha'(G) +\alpha(G) - n = \alpha(G).\] Therefore, there are at most $n-2(\alpha'(G)-i)$ vertices which can be added to $X_v$. Hence, $|X| = |X_v| + |X_e| \leq n-2(\alpha'(G)-i) + \alpha'(G)-i = n-\alpha'(G) + i$. Observe that this is increasing with $i$ as $i$ goes from 0 to $m'$. Hence, for all $i \in \mathbb{N}_0$ with $0 \leq i \leq m'$ and for all $X \in \mathcal{X}_{G_{\alpha'(G)-i}}$, $|X| \leq n-\alpha'(G) + m'$

Now consider $X \in \mathcal{X}_{G_{\alpha'(G)-(m'+1)}}$, write $X$ as $X = X_v \cup X_e$. By definition of $\mathcal{X}_{G_{\alpha'(G)-(m'+1)}}$, $|X_e| = \alpha'(G)-(m'+1)$. This means that $2(\alpha'(G)-(m'+1))$ vertices appear in the set as end vertices. However, $m'$ has the property that $m' \leq \frac{\alpha(G)-n + 2\alpha'(G)}{2}$ and $m'+1 > \frac{\alpha(G)-n + 2\alpha'(G)}{2}$. Therefore \[n-2(\alpha'(G)-(m'+1)) = n - 2\alpha'(G) + 2(m'+1) > n - 2\alpha'(G) + 2\left(\frac{\alpha(G)-n + 2\alpha'(G)}{2}\right)\]\[= n-2\alpha'(G) + 2\alpha'(G) + \alpha(G) - n = \alpha(G).\] However, we can not add more than $\alpha(G)$ vertices to the $X_v$. Therefore for $X \in \mathcal{X}_{G_{\alpha'(G)-i}}$ for $i>m'$, we can only add at most $\alpha(G)$ more vertices to the set $X_v$.

Hence, for $X \in \mathcal{X}_{G_{\alpha'(G)-i}}$ with $i>m'$ and writing $X$ as $X = X_v \cup X_e$ where $|X_e| = \alpha'(G)-i$. It follows that $|X| = |X_v| + |X_e| \leq \alpha(G) + \alpha'(G)-i$, which is decreasing as $i$ increases. It follows for all $i \in \mathbb{N}_0$ with  $m'+1 \leq i \leq \alpha'(G)$ and for all $X \in \mathcal{X}_{G_{\alpha'(G)-i}}$, $|X| \leq \alpha(G) + \alpha'(G) -(m'+1)$

Now, $\frac{\alpha(G)-n + 2\alpha'(G)}{2} \leq m'+\frac{1}{2}$ which implies that $\frac{\alpha(G)-n+ 2\alpha'(G)-1}{2} \leq m'$ and so $\alpha(G)-n+ 2\alpha'(G)-1 \leq 2m'$. Hence, $\alpha'(G) - (m'+1) + \alpha(G) \leq n-\alpha'(G) + m'$. Therefore, for all $i \in \mathbb{N}_0$ where $0 \leq i \leq \alpha'(G)$ and for all $X \in \mathcal{X}_{G_{\alpha'(G)-i}}$, $|X| \leq n-\alpha'(G) + m'$, and so $\alpha''(G) \leq n-\alpha'(G) + m'$.
\end{proof}

\section{How good is the estimate?}
\label{howgood}

The next question to consider is how good the bound is that Theorem \ref{mainbound} provides.

\begin{lemma}
\label{sharpbound}
The bound in Theorem \ref{mainbound} is sharp for all $n$.
\end{lemma}
\begin{proof}
Consider the complete graph $K_n$ where $n$ is odd. By definition of $K_n$, $\alpha(K_n) = 1$ and $\alpha'(K_n) = \floor*{\frac{n}{2}} = \frac{n-1}{2}$ as $n$ is odd. Therefore by Theorem \ref{mainbound}, \[\alpha''(G) \leq n+ \floor*{\frac{1-n+n-1}{2}} - \frac{n-1}{2} = \frac{1}{2}(n+1) = \frac{1}{2}(n-1)+1.\] We can construct a total independent set of this size by first taking any maximal edge independent set. This contains $n-1$ distinct vertices as end vertices of the edges in this edge independent set and so we can then take the total independent set to be the $\frac{n-1}{2}$ edges in the edge independent set and the vertex left over. This is a total independent set of size $\frac{1}{2}(n-1)+1$ so $\alpha''(G) = \frac{1}{2}(n-1)+1$ in this case.

For $n$ even, $\alpha(K_n) = 1$ and $\alpha'(K_n) = \frac{n}{2}$. Therefore by Theorem \ref{mainbound}, \[\alpha''(G) \leq n + \floor*{\frac{1-n+n}{2}} - \frac{n}{2} = \frac{n}{2}.\] We can construct a total independent set of this size by taking any maximal edge independent set.

\end{proof}

However, it is not always a sharp bound. Consider the graph $K_{2,3}$. It follows from the definition of $K_{2,3}$ that $\alpha(K_{2,3}) = 3$ and $\alpha'(G) = 2$. Therefore by Theorem \ref{mainbound},  $\alpha''(G) \leq 5 + \floor*{\frac{3-5+4}{2}} - 2 = 4$. However, we can not construct a total independent set of this size. If we have a total independent set containing 2 edges, we can only at most add one extra vertex to the total independent set as 4 vertices are end vertices of the edges. If we have a total independent set containing no edges, we can have at most 3 vertices in the total independent set as $\alpha(K_{2,3}) = 3$. 

This leaves the case of a total independent set containing 1 edge. Let $V(K_{2,3}) = A \cup B$ be the partition of the vertices with $|A| = 2$ and $|B| = 3$. The edge contained in the edge independent set is of the form $x \cdot y$ where $x \in A$ and $y \in B$. Therefore, there are three vertices left you could add to the total independent set. However, it is clear you can not add all 3 as the vertex in $A \setminus \{x\}$ is adjacent to both the vertices in $B$. Therefore, we can only add the two vertices in $B \setminus \{y\}$ so the total independent set has size 3. This implies that $\alpha''(G) = 3 < 4$ and the bound is not an equality. 

\[\begin{tikzpicture} 
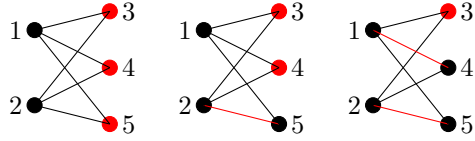
          \node at (0,-1.5) {Three total independent sets (coloured red) for $K_{2,3}$.};

          \node at (-0.75,-0.5) {2};
          \node at (-0.75,0.5) {1};
          \node at (0.75,-0.75) {5};
          \node at (0.75,0) {4};
          \node at (0.75,0.75) {3};
          \draw [fill] (-0.5,-0.5) circle [radius=0.1];
		 \draw [fill] (-0.5,0.5) circle [radius=0.1];
		 \draw [fill] (0.5,-0.75) circle [radius=0.1];
		 \draw [fill,color=red] (0.5,0) circle [radius=0.1];
          \draw [fill,color=red] (0.5,0.75) circle [radius=0.1];
          \draw[color=red] (-0.5,-0.5) -- (0.5,-0.75);
          \draw (-0.5,-0.5) -- (0.5,0);
          \draw (-0.5,-0.5) -- (0.5,0.75);
          \draw (-0.5,0.5) -- (0.5,-0.75);
          \draw (-0.5,0.5) -- (0.5,0);
          \draw (-0.5,0.5) -- (0.5,0.75);

          \node at (-3,-0.5) {2};
          \node at (-3,0.5) {1};
          \node at (-1.5,-0.75) {5};
          \node at (-1.5,0) {4};
          \node at (-1.5,0.75) {3};
          \draw [fill] (-2.75,-0.5) circle [radius=0.1];
		 \draw [fill] (-2.75,0.5) circle [radius=0.1];
		 \draw [fill, color = red] (-1.75,-0.75) circle [radius=0.1];
		 \draw [fill, color = red] (-1.75,0) circle [radius=0.1];
          \draw [fill, color = red] (-1.75,0.75) circle [radius=0.1];
          \draw (-2.75,-0.5) -- (-1.75,-0.75);
          \draw (-2.75,-0.5) -- (-1.75,0);
          \draw (-2.75,-0.5) -- (-1.75,0.75);
          \draw (-2.75,0.5) -- (-1.75,-0.75);
          \draw (-2.75,0.5) -- (-1.75,0);
          \draw (-2.75,0.5) -- (-1.75,0.75);

         \node at (1.5,-0.5) {2};
          \node at (1.5,0.5) {1};
          \node at (3,-0.75) {5};
          \node at (3,0) {4};
          \node at (3,0.75) {3};
          \draw [fill] (1.75,-0.5) circle [radius=0.1];
		 \draw [fill] (1.75,0.5) circle [radius=0.1];
		 \draw [fill] (2.75,-0.75) circle [radius=0.1];
		 \draw [fill] (2.75,0) circle [radius=0.1];
          \draw [fill,color=red] (2.75,0.75) circle [radius=0.1];
          \draw[color=red] (1.75,-0.5) -- (2.75,-0.75);
          \draw (1.75,-0.5) -- (2.75,0);
          \draw (1.75,-0.5) -- (2.75,0.75);
          \draw (1.75,0.5) -- (2.75,-0.75);
          \draw[color=red] (1.75,0.5) -- (2.75,0);
          \draw (1.75,0.5) -- (2.75,0.75);
          
  \end{tikzpicture}\]

We now show that the estimate in Theorem \ref{mainbound} is the same as or better than the result in Lemma \ref{basicbound}.

\begin{lemma}
The result in Theorem \ref{mainbound} is equal to or better than the result in Lemma \ref{basicbound}.
\end{lemma}
\begin{proof}
We have that $n+m'-\alpha'(G) \leq n+2m'-\alpha'(G) \leq n+\alpha(G) - n +2\alpha'(G) - \alpha'(G) = \alpha(G) + \alpha'(G)$.
\end{proof}

We can prove easily when this is a strict inequality based on the value of $\frac{\alpha(G)-n + 2\alpha'(G)}{2}$.

\begin{lemma}
If $\frac{\alpha(G)-n + 2\alpha'(G)}{2} > 0$ then the result in Theorem \ref{mainbound} is strictly better than the result in Lemma \ref{basicbound}.
\end{lemma}
\begin{proof}
Case 1: $\frac{\alpha(G)-n + 2\alpha'(G)}{2} \geq 1$. In this case, $\frac{\alpha(G)-n + 2\alpha'(G)}{2} \geq 1$ which implies that  $m' \geq 1$. Therefore, \[n+m'-\alpha'(G) < n+2m'-\alpha'(G) \leq n+\alpha(G) - n +2\alpha'(G) - \alpha'(G) = \alpha(G) + \alpha'(G).\]

Case 2: $\frac{\alpha(G)-n + 2\alpha'(G)}{2} = \frac{1}{2}$. In this case, $\frac{\alpha(G)-n + 2\alpha'(G)}{2} = \frac{1}{2}$ which implies that $m' = 0$. We have that $0 < \frac{\alpha(G)-n + 2\alpha'(G)}{2}$ which implies that $0 < \alpha(G)-n + 2\alpha'(G)$, and so $n-\alpha'(G) < \alpha(G) + \alpha'(G)$.
\end{proof}

We can show that this is in general not true when $\frac{\alpha(G)-n + 2\alpha'(G)}{2} = 0$. Consider the graph $K_3$. It follows from the definition of $K_3$ that $\alpha(K_3) = 1$ and $\alpha'(K_3) = 1$. This satisfies the condition $\frac{\alpha(K_3)-n + 2\alpha'(K_3)}{2} = \frac{1-3 + 2}{2} = 0$ and we can see that $\alpha(K_3) + \alpha'(K_3) = 2$ and $n - \alpha'(G) = 2$. The same is true for any $K_n$ where $n$ is odd. When $n$ is even, $\alpha(G) = 1$, and $\alpha'(G) = \frac{n}{2}$, and so $\frac{\alpha(G)-n+2\alpha'(G)}{2} = \frac{1}{2}$. 

We can characterise simple, connected, and finite Hamiltonian graphs $G$ with the property that $\frac{\alpha(G)-n + 2\alpha'(G)}{2} = 0$. Observe that $\frac{\alpha(G)-n + 2\alpha'(G)}{2} = 0$ if and only if $\alpha(G) + 2\alpha'(G) = n$, and $\floor*{\frac{n}{2}} = \frac{n}{2}$ if $n$ is even, and $\floor*{\frac{n}{2}} = \frac{n-1}{2}$ if $n$ is odd. By Lemma \ref{Hamiltonianedgeindep}, $\alpha'(G) = \floor*{\frac{n}{2}}$. Therefore, if $n$ is even, $\alpha(G) +2\alpha'(G) = \alpha(G) +n$. This implies that if $\alpha(G) + 2\alpha'(G) = n$, then $\alpha(G) + n = n$, and so $\alpha(G)=0$ which is impossible if $G$ is non-empty. if $n$ is odd, $\alpha(G) +2\alpha'(G) = \alpha(G) +n-1$. This implies that if $\alpha(G) + 2\alpha'(G) = n$, then $\alpha(G) + n-1 = n$, and so $\alpha(G)=1$. A graph has $\alpha(G) = 1$ if and only if $G$ is the complete graph, since any missing edge gives rise to an independent set of size $2$. It would be interesting to know if there are other families of graphs for which we can characterise when $\frac{\alpha(G)-n + 2\alpha'(G)}{2} = 0$.

\section{Structure of maximal total independent set}
\label{structure}

Theorem \ref{mainbound} considers every possible structure of a total independent set and produces a bound for each such set. We can adapt this proof to provide the possible structures of a maximal total independent set. 

\begin{corollary}
Let $G$ be a simple, connected and finite graph of order $n$. The number of edges $e$ in a total independent set of size $k$ must satisfy $k-\alpha(G) \leq e \leq n - k$.
\end{corollary}
\begin{proof}
From the proof of Theorem \ref{mainbound}, we can see that for a total independent set with $e = \alpha'(G)-i$ edges in it with $i \leq m'$, the size of a total independent set can be at most $n-\alpha'(G)+i$. Therefore, we want to consider the cases where $n-\alpha'(G)+i \geq k$. This implies that $i \geq k + \alpha'(G) - n$ and so $-i \leq n-k-\alpha'(G)$. Hence $e = \alpha'(G) - i \leq n-k$.

Now for $i > m'$, from the proof of Theorem \ref{mainbound}, a total independent set containing $e = \alpha'(G)-i$ edges can have size at most $\alpha'(G)-i +\alpha(G)$. Similar to the previous case, we want to consider the cases where $\alpha'(G)-i +\alpha(G) \geq k$. This implies that $i \leq \alpha'(G)+\alpha(G)-k$ and so $-i \geq k - \alpha'(G) - \alpha(G)$. Hence $e = \alpha'(G) - i \geq k - \alpha(G)$.

Combining the two, we obtain $k - \alpha(G) \leq e \leq n-k$.
\end{proof}

We note by letting $k = \alpha''(G)$, we immediately obtain the following result.

\begin{corollary}
\label{structurecor}
Let $G$ be a simple,connected and finite graph of order $n$ with total independence number $\alpha''(G)$. 
The number of edges $e$ in a largest total independent set of size $\alpha''(G)$ must satisfy $\alpha''(G)-\alpha(G) \leq e \leq n - \alpha''(G)$.
\qedno
\end{corollary}

We can use Corollary \ref{structurecor} to easily find another bound on $\alpha''(G)$ in terms of $\alpha(G)$.

\begin{corollary}
\label{otherbound}
Let $G$ be a simple connected and finite graph of order $n$ with total independence number $\alpha''(G)$. 
Then $\alpha''(G) \leq \frac{n + \alpha(G)}{2}$.
\end{corollary}
\begin{proof}
By Corollary \ref{structurecor}, $\alpha''(G)-\alpha(G) \leq e \leq n - \alpha''(G)$ which implies that $\alpha''(G)-\alpha(G) \leq n - \alpha''(G)$. Rearranging this, we obtain $\alpha''(G) \leq \frac{n + \alpha(G)}{2}$.
\end{proof}

The total independence number is a graph invariant which has not been studied much in the literature, and so there are still some unanswered questions. It is known that there are relations between the total independence number and the independence number, edge independence number, arboricity, edge arboricity and edge domination number. 

\begin{question}
Are there relations between the total independence number with any other invariants?
\end{question}

Lemma \ref{sharpbound} shows that the bound in Theorem \ref{mainbound} is sharp. The bound is sharp for the complete graph since the degree of each vertex is high. 

\begin{question}
    Does there exist a simple, connected and finite graph $G$ of order $n$ and maximum degree $d$ such that the bound in Theorem \ref{mainbound} is sharp for all $n$ and all $d$? 
\end{question}

\bibliographystyle{amsalpha}

\end{document}